\documentclass[12pt]{article}

\usepackage{latexsym,amsmath,amscd,amssymb,graphics,float}
\usepackage{enumerate}

\usepackage{graphicx,lscape}

\usepackage[colorlinks]{hyperref}
\usepackage{url}

\begin{document}

\newenvironment{proof}[1][Proof]{\textbf{#1.} }{\ \rule{0.5em}{0.5em}}

\newtheorem{theorem}{Theorem}[section]
\newtheorem{definition}[theorem]{Definition}
\newtheorem{lemma}[theorem]{Lemma}
\newtheorem{remark}[theorem]{Remark}
\newtheorem{proposition}[theorem]{Proposition}
\newtheorem{corollary}[theorem]{Corollary}
\newtheorem{example}[theorem]{Example}

\numberwithin{equation}{section}
\newcommand{\ep}{\varepsilon}
\newcommand{\R}{{\mathbb  R}}
\newcommand\C{{\mathbb  C}}
\newcommand\Q{{\mathbb Q}}
\newcommand\Z{{\mathbb Z}}
\newcommand{\N}{{\mathbb N}}

\newcommand{\bfi}{\bfseries\itshape}

\newsavebox{\savepar}
\newenvironment{boxit}{\begin{lrbox}{\savepar}
\begin{minipage}[b]{15.5cm}}{\end{minipage}\end{lrbox}
\fbox{\usebox{\savepar}}}

\title{{\bf A global geometric decomposition of vector fields and applications to topological conjugacy}}
\author{R\u{a}zvan M. Tudoran}

\date{}
\maketitle \makeatother

\begin{abstract}
We give a global geometric decomposition of continuously differentiable vector fields on $\mathbb{R}^n$. More precisely, given a vector field of class $\mathcal{C}^{1}$ on $\mathbb{R}^{n}$, and a geometric structure on $\mathbb{R}^n$, we provide a unique global decomposition of the vector field as the sum of a left (right) gradient--like vector field (naturally associated to the geometric structure) with potential function vanishing at the origin, and a vector field which is left (right) orthogonal to the identity, with respect to the geometric structure. As application, we provide a criterion to decide topological conjugacy of complete vector fields of class $\mathcal{C}^1$ on $\mathbb{R}^{n}$ based on topological conjugacy of the corresponding parts given by the associated geometric decompositions.
\end{abstract}

\medskip



\section{Introduction}
\label{section:one}

When thinking about decompositions of vector fields, the first example that crosses our minds is the famous Helmholtz--Hodge decomposition (see e.g. \cite{helmholtz}, \cite{shwarz}), also known as the \textit{fundamental theorem of vector analysis}, which provides a method to decompose a $3-$dimensional vector field (defined on some bounded domain, and satisfying suitable regularity conditions), into the sum of a gradient vector field and a divergence--free vector field. Starting from the original version, many extensions and generalizations of this decomposition have been given in various contexts, see e.g. \cite{deriaz}, \cite{ort}, \cite{shwarz}. Another interesting decomposition such that one of its summands is a gradient vector field, was given by Presnov \cite{presnov}. More precisely, he provides a method to decompose (uniquely) any continuously differentiable vector field globally defined on $\mathbb{R}^n$, as a sum of a gradient vector field (with potential function vanishing at the origin) and a vector field orthogonal to the identity, at each point of $\mathbb{R}^n$. The main part of the present work is concerned with a geometric generalization of Presnov's decomposition of continuously differentiable vector fields. More precisely, given a vector field of class $\mathcal{C}^{1}$ on $\mathbb{R}^{n}$, and an arbitrary fixed geometric structure on $\mathbb{R}^n$ (e.g. Euclidean, symplectic, Minkowski), we provide a unique global decomposition of the vector field as a sum of a left (right) gradient--like vector field naturally associated to the geometric structure (e.g. gradient vector field, Hamiltonian vector field, Minkowski gradient vector field) with potential function vanishing at the origin, and a vector field which is left (right) orthogonal to the identity, with respect to the geometric structure. As application, we provide a criterion to decide topological conjugacy of complete vector fields of class $\mathcal{C}^1$ on $\mathbb{R}^{n}$ based on topological conjugacy of the corresponding parts given by the associated geometric decompositions.

The structure of the paper is as follows. In the second section we prepare the geometric settings of the problem, by introducing the notion of geometric structure (i.e. a general nondegenerate real bilinear form on $\mathbb{R}^{n}$), whose geometry encompasses some of the classical geometries of $\mathbb{R}^{n}$ (e.g. Euclidean geometry, symplectic geometry, Minkowski geometry). Associated to an arbitrary fixed geometric structure, we define two gradient--like vector fields (i.e. the left/right gradient--like vector field) which extends naturally the notion of gradient vector field (from  Euclidean geometry), Hamiltonian vector field (from symplectic geometry), and Minkowski gradient vector field (from Minkowski geometry). Next, we prove the global Poincar\'e lemma in the context of gradient--like vector fields associated to a general geometric structure. The third section contains the main result of this work, which provides a global geometric decomposition of the vector fields of class $\mathcal{C}^{1}$ on $\mathbb{R}^{n}$ with respect to an arbitrary fixed geometric structure on $\mathbb{R}^{n}$. More precisely, given a vector field of class $\mathcal{C}^{1}$ on $\mathbb{R}^{n}$, and a geometric structure, we provide a unique global decomposition of the vector field as the sum of a left (right) gradient--like vector field (naturally associated to the geometric structure) with potential function vanishing at the origin, and a vector field which is left (right) orthogonal to the identity, with respect to the geometric structure. Then we apply this decomposition for two vector fields which generate some well known dynamical systems, namely, the Lotka--Volterra system from biology, and the Rikitake system from geophysics. As a consequence, we obtain that the Rikitake system is actually a Minkowski gradient system, if certain parameter vanishes. In the last section we present a theoretical application of the main result of this paper, in order to analyze the topological conjugacy of continuously differentiable vector fields in $\mathbb{R}^{n}$. More precisely, we provide a criterion to decide topological conjugacy of complete vector fields of class $\mathcal{C}^1$ on $\mathbb{R}^{n}$ based on topological conjugacy of the corresponding parts given by the associated geometric decompositions. This criterion is a geometric generalization of the main result from \cite{anas}.

\section{Geometric structures on $\mathbb{R}^{n}$ and Poincar\'e's lemma}

As the main purpose of this work is to provide a global geometric decomposition of the vector fields on $\mathbb{R}^{n}$, we introduce first the geometric settings of the problem. In order to do that, let us notice that the common starting point in defining some of the main classical geometric structures on $\mathbb{R}^{n}$ (e.g. Euclidean, symplectic, Minkowski) is to chose some nondegenerate real bilinear form satisfying certain properties specific to each geometric structure apart. Thus, a natural geometric framework which incorporates all of the above mentioned geometric structures, is given by the choice of a general nondegenerate real bilinear form on $\mathbb{R}^{n}$.

\begin{definition}
Let us call a \textbf{geometric structure} on $\mathbb{R}^{n}$ each nondegenerate real bilinear form, $b:\mathbb{R}^{n} \times \mathbb{R}^{n}\rightarrow \mathbb{R}$. The pair $(\mathbb{R}^{n},b)$, where $b$ is a geometric structure on $\mathbb{R}^{n}$, will be denoted by $\mathbb{R}^{n}_{b}$.
\end{definition}

\begin{example} Let $b$ be a geometric structure on $\mathbb{R}^{n}$. Then the following assertions hold.
\begin{enumerate}
\item[(i)] If $b$ is symmetric and positive definite, then $\mathbb{R}^{n}_{b}$ is an \textit{Euclidean vector space}. 
\item[(ii)] If $n$ is even and $b$ is skew--symmetric, then $\mathbb{R}^{n}_{b}$ is a \textit{symplectic vector space}. Recall that symplectic forms (i.e. skew--symmetric nondegenerate real bilinear forms) on $\mathbb{R}^{n}$ exist only if $n$ is even.
\item[(iii)] If $b$ is symmetric with the signature $(+,\dots,+,-)$, then $\mathbb{R}^{n}_{b}$ is a \textit{Minkowski vector space}.
\end{enumerate} 
\end{example}

Next result is a direct consequence of the nondegeneracy property of geometric structures on $\mathbb{R}^{n}$, and provides a compatibility relation between a general geometric structure and a fixed one, in this case, the canonical inner product on $\mathbb{R}^{n}$. 

\begin{proposition}
Let $b$ be a geometric structure on $\mathbb{R}^{n}$. Then there exists a unique invertible linear map $B\in\operatorname{Aut}(\mathbb{R}_{b}^{n})$ such that 
\begin{equation*}
\langle \mathbf{x},\mathbf{y} \rangle = b(\mathbf{x}, B\mathbf{y}), ~~\forall \mathbf{x},\mathbf{y}\in\mathbb{R}^{n}_{b},
\end{equation*}
 where $\langle \cdot,\cdot \rangle$ stands for the canonical inner product on $\mathbb{R}^{n}$. The pair $(b,B)$ will be called a \textbf{geometric pair} on $\mathbb{R}^{n}$.
\end{proposition}

Note that naturally associated to a geometric pair $(b,B)$, there exists an invertible linear map, $B^{\star}$, the (left) \textbf{adjoint} of $B$ with respect to $b$, uniquely defined by the relation
\begin{equation*}
b(B^{\star}\mathbf{x}, \mathbf{y})= b(\mathbf{x}, B\mathbf{y}), ~~\forall \mathbf{x},\mathbf{y}\in\mathbb{R}^{n}_{b}.
\end{equation*}

Recall that $\left(B^{\star}\right)^{\star}= B$ if $b$ is symmetric or skew--symmetric, although for a general geometric structure $b$ it could happen that $\left(B^{\star}\right)^{\star}\neq B$. Let us give now some properties of the linear map $B^{\star}$, associated to a geometric pair $(b,B)$.

\begin{proposition}
Let $b$ be a geometric structure on $\mathbb{R}^{n}$, and let $(b,B)$ be the associated geometric pair. Then the following statements hold.
\begin{enumerate}
\item[(i)] $B^{\star}= B^{\top}$, where $B^{\top}$ stands for the adjoint of $B$ with respect to $\langle\cdot,\cdot\rangle$.
\item[(ii)] If $b$ is symmetric then $B^{\star}=B$.
\item[(iii)] If $b$ is skew--symmetric then $B^{\star}= - B$.
\end{enumerate}
\end{proposition}
\begin{proof}
\begin{enumerate}
\item[(i)] In order to prove that $B^{\star}= B^{\top}$, we shall show that $b(B^{\star}\mathbf{x}, \mathbf{y})=b(B^{\top}\mathbf{x}, \mathbf{y}), ~\forall \mathbf{x},\mathbf{y}\in\mathbb{R}^{n}_{b}$.  This follows easily using the definition of a geometric pair. Indeed, for all $\mathbf{x},\mathbf{y}\in\mathbb{R}^{n}_{b}$ we obtain successively
\begin{align*}
b(B^{\top}\mathbf{x}, \mathbf{y})=b(B^{\top}\mathbf{x}, B B^{-1}\mathbf{y})=\langle B^{\top}\mathbf{x},B^{-1}\mathbf{y}\rangle =\langle \mathbf{x},B B^{-1}\mathbf{y}\rangle =\langle \mathbf{x},\mathbf{y}\rangle = b(B^{\star}\mathbf{x}, \mathbf{y}).
\end{align*}
\item[(ii)] As $b$ is symmetric we have that $(B^{\star})^{\star}=B$. Now, in order to prove that $B^{\star}= B$, we shall show that $b(B^{\star}\mathbf{x}, \mathbf{y})=b(B\mathbf{x}, \mathbf{y}), ~\forall \mathbf{x},\mathbf{y}\in\mathbb{R}^{n}_{b}$. This follows easily using the definition of a geometric pair and the symmetry of $b$. Indeed, for all $\mathbf{x},\mathbf{y}\in\mathbb{R}^{n}_{b}$ we obtain successively
\begin{align*}
b(B\mathbf{x}, \mathbf{y})= b((B^{\star})^{\star}\mathbf{x},\mathbf{y})=  b(\mathbf{x},B^{\star}\mathbf{y})= b(B^{\star}\mathbf{y},\mathbf{x})=\langle \mathbf{y},\mathbf{x}\rangle =\langle \mathbf{x},\mathbf{y}\rangle =b(B^{\star}\mathbf{x}, \mathbf{y}).
\end{align*}
\item[(iii)] As $b$ is skew--symmetric we have that $(B^{\star})^{\star}=B$. In order to prove that $B^{\star}= - B$, we shall show that $b(B^{\star}\mathbf{x}, \mathbf{y})= - b(B\mathbf{x}, \mathbf{y}), ~\forall \mathbf{x},\mathbf{y}\in\mathbb{R}^{n}_{b}$. This time we will use the definition of a geometric pair and the skew-symmetry of $b$. Indeed, for all $\mathbf{x},\mathbf{y}\in\mathbb{R}^{n}_{b}$ we obtain successively
\begin{align*}
b(B\mathbf{x}, \mathbf{y})&= b((B^{\star})^{\star}\mathbf{x},\mathbf{y})=  b(\mathbf{x},B^{\star}\mathbf{y})= - b(B^{\star}\mathbf{y},\mathbf{x})= - \langle \mathbf{y},\mathbf{x}\rangle = - \langle \mathbf{x},\mathbf{y}\rangle\\
& = - b(B^{\star}\mathbf{x}, \mathbf{y}).
\end{align*}
\end{enumerate}
\end{proof}

\begin{remark}
Note that even if $B^{\star}= B^{\top}$, this equality is not necessary true for each invertible linear map defined on $\mathbb{R}^{n}_{b}$. This is mainly due to the fact that the map $B$ together with the geometric structure $b$ form the geometric pair $(b,B)$. Thus, the equality $B^{\star}= B^{\top}$ does not imply in general that $(B^{\star})^{\star}=B$. Actually, one can easily show that given a geometric pair $(b,B)$, the equality $(B^{\star})^{\star}=B$ holds true if and only if $B$ is $b-$normal, i.e. $BB^{\star}=B^{\star}B$.
\end{remark}

Given a geometric structure $b$ on $\mathbb{R}^n$, we introduce two gradient--like operators, namely, the left--gradient, $\nabla^{L}_{b}$, and respectively, the right--gradient, $\nabla^{R}_{b}$, uniquely defined for any function $F\in\mathcal{C}^{1}(\mathbb{R}^{n}_{b},\mathbb{R})$ by the relation
\begin{equation}\label{lg}
b(\nabla^{L}_{b}F(\mathbf{x}),\mathbf{v})=\mathrm{d}F(\mathbf{x})\cdot \mathbf{v}, ~ \forall \mathbf{x},\mathbf{v}\in\mathbb{R}^{n}_{b},
\end{equation}
and respectively 
\begin{equation}\label{rg}
b(\mathbf{v},\nabla^{R}_{b}F(\mathbf{x}))=\mathrm{d}F(\mathbf{x})\cdot \mathbf{v}, ~ \forall \mathbf{x},\mathbf{v}\in\mathbb{R}^{n}_{b}.
\end{equation}

Note that in the case when $b$ is symmetric, $\nabla^{L}_{b}=\nabla^{R}_{b}$. Moreover, when $b$ equals the canonical inner product on $\mathbb{R}^n$, we have the equality $\nabla^{L}_{b}=\nabla^{R}_{b}=\nabla$, where $\nabla$ stands for the classical gradient operator on $\mathbb{R}^n$, defined for any function $F\in\mathcal{C}^{1}(\mathbb{R}^{n}_{b},\mathbb{R})$ by the relation
\begin{equation}\label{cg}
\langle\nabla F(\mathbf{x}),\mathbf{v}\rangle=\mathrm{d}F(\mathbf{x})\cdot \mathbf{v}, ~ \forall \mathbf{x},\mathbf{v}\in\mathbb{R}^{n}_{b}.
\end{equation}

On the other hand, in the case when $b$ is skew--symmetric, it follows that $\nabla^{L}_{b}= - \nabla^{R}_{b}$. 

Consequently, for $b$ symmetric or skew--symmetric, we have simple relations between $\nabla^{L}_{b}$ and $\nabla^{R}_{b}$ which can be unified in the following notation:
$$\nabla_{b}H:=\left\{\begin{array}{ll}\nabla^{L}_{b}H=\nabla^{R}_{b}H, & \mbox{if b is symmetric} \\
\nabla^{L}_{b}H=-\nabla^{R}_{b}H, & \mbox{if b is skew--symmetric}
\end{array}\right..$$

Even for a general geometric structure $b$, the gradient--like operators, $\nabla^{L}_{b}$, $\nabla^{R}_{b}$ are also related in a natural way, as we can see in the following proposition.

\begin{proposition}
Let $b$ be a geometric structure on $\mathbb{R}^{n}$, and let $(b,B)$ be the associated geometric pair. For every $F\in\mathcal{C}^{1}(\mathbb{R}^{n}_{b},\mathbb{R})$, the following relations hold true.
\begin{enumerate}
\item[(i)] $\nabla^{R}_{b}F=B \nabla F$,
\item[(ii)] $\nabla^{L}_{b}F=B^{\star}\nabla F$,
\item[(iii)] $\nabla^{R}_{b}F=B(B^{\star})^{-1} \nabla^{L}_{b}F$.
\end{enumerate}
\end{proposition}
\begin{proof}
\begin{enumerate}
\item[(i)] Using the relation \eqref{rg} and the fact that $(b,B)$ is a geometric pair, we obtain that
\begin{equation*}
\mathrm{d}F(\mathbf{x})\cdot \mathbf{v}=b(\mathbf{v},\nabla^{R}_{b}F(\mathbf{x}))=\langle(B^{\star})^{-1}\mathbf{v},\nabla^{R}_{b}F(\mathbf{x})\rangle, \forall \mathbf{x},\mathbf{v}\in\mathbb{R}^{n}_{b}.
\end{equation*}

The above relation together with the identity \eqref{cg} imply that
\begin{equation}\label{rl2}
\langle(B^{\star})^{-1}\mathbf{v},\nabla^{R}_{b}F(\mathbf{x})\rangle=\langle\mathbf{v},\nabla F(\mathbf{x})\rangle, \forall \mathbf{x},\mathbf{v}\in\mathbb{R}^{n}_{b}.
\end{equation}

As $B^{\star}=B^{\top}$, the relation \eqref{rl2} becomes
\begin{equation*}
\langle\mathbf{v},B^{-1}\nabla^{R}_{b}F(\mathbf{x})\rangle=\langle\mathbf{v},\nabla F(\mathbf{x})\rangle, \forall \mathbf{x},\mathbf{v}\in\mathbb{R}^{n}_{b},
\end{equation*}
and so, we get $\nabla^{R}_{b}F=B \nabla F$.
\item[(ii)] The proof is similar to the one of item $(i)$, this time using the relation \eqref{lg} and the fact that $(b,B)$ is a geometric pair.
\item[(iii)] The proof follows directly from $(i)$ and $(ii)$.
\end{enumerate}
\end{proof}

In the following we present the global Poincar\'e lemma in the case of left--gradient and right--gradient vector fields associated to a general geometric structure $b$ on $\mathbb{R}^n$. As particular cases, we recover the classical characterizations of the gradient and the Hamiltonian vector fields.

\begin{theorem}\label{Poincare}
Let $X\in\mathfrak{X}(\mathbb{R}^{n})$ be a vector field on class $\mathcal{C}^{1}$ on $\mathbb{R}^{n}$. Let $b$ be a geometric structure on $\mathbb{R}^{n}$, and let $(b,B)$ be the associated geometric pair. Then the following assertions hold true.
\begin{enumerate}
\item[(i)] $X=\nabla^{L}_{b}H$ for some $H\in\mathcal{C}^{2}(\mathbb{R}_{b}^{n},\mathbb{R})$ if and only if
\begin{equation}\label{relunu}
(\mathrm{D}X(\mathbf{x}))^{\top}B^{-1}=(B^{\star})^{-1}\mathrm{D}X(\mathbf{x}), ~ \forall \mathbf{x}\in\mathbb{R}_{b}^{n}.
\end{equation}
Moreover, if $X$ satisfies \eqref{relunu} then the solution of the equation $X=\nabla^{L}_{b}H$ is given by $$H(\mathrm{x})=\int_{0}^{1}b(X(t\mathbf{x}),\mathbf{x})\mathrm{d}t+ constant, ~ \forall \mathbf{x}\in\mathbb{R}_{b}^{n}.$$
\item[(ii)] $X=\nabla^{R}_{b}H$ for some $H\in\mathcal{C}^{2}(\mathbb{R}_{b}^{n},\mathbb{R})$ if and only if
\begin{equation}\label{reldoi}
(\mathrm{D}X(\mathbf{x}))^{\top}(B^{\star})^{-1}=B^{-1}\mathrm{D}X(\mathbf{x}), ~ \forall \mathbf{x}\in\mathbb{R}_{b}^{n}.
\end{equation}
Moreover, if $X$ satisfies \eqref{reldoi} then the solution of the equation $X=\nabla^{R}_{b}H$ is given by  $$H(\mathrm{x})=\int_{0}^{1}b(\mathbf{x},X(t\mathbf{x}))\mathrm{d}t+ constant, ~ \forall \mathbf{x}\in\mathbb{R}_{b}^{n}.$$
\item[(iii)] If $b$ is symmetric, then $X=\nabla_{b}H$ for some $H\in\mathcal{C}^{2}(\mathbb{R}_{b}^{n},\mathbb{R})$ if and only if
\begin{equation}\label{reltrei}
(\mathrm{D}X(\mathbf{x}))^{\top}B^{-1}=B^{-1}\mathrm{D}X(\mathbf{x}), ~ \forall \mathbf{x}\in\mathbb{R}_{b}^{n}.
\end{equation}
Moreover, if $X$ satisfies \eqref{reltrei} then the solution of the equation $X=\nabla_{b}H$ is given by $$H(\mathrm{x})=\int_{0}^{1}b(X(t\mathbf{x}),\mathbf{x})\mathrm{d}t+ constant, ~ \forall \mathbf{x}\in\mathbb{R}_{b}^{n}.$$
\item[(iv)] If $n$ is even and $b$ is skew--symmetric (i.e. $b$ is a symplectic form), then $X=\nabla_{b}H$ for some $H\in\mathcal{C}^{2}(\mathbb{R}_{b}^{n},\mathbb{R})$ if and only if
\begin{equation}\label{relpatru}
(\mathrm{D}X(\mathbf{x}))^{\top}B^{-1} + B^{-1}\mathrm{D}X(\mathbf{x}) = O_{n}, ~ \forall \mathbf{x}\in\mathbb{R}_{b}^{n}.
\end{equation}
Moreover, if $X$ satisfies \eqref{relpatru} then the solution of the equation $X=\nabla_{b}H$ is given by $$H(\mathrm{x})=\int_{0}^{1}b(X(t\mathbf{x}),\mathbf{x})\mathrm{d}t+ constant, ~ \forall \mathbf{x}\in\mathbb{R}_{b}^{n}.$$

Here $\nabla_{b}H$ is precisely the Hamiltonian vector field of $H$ with respect to the symplectic form $b$, i.e. $\nabla_{b}H = X_H$.
\end{enumerate}
\end{theorem}
\begin{proof}
\begin{enumerate}
\item[(i)] As $\nabla^{L}_{b}F = B^{\star}\nabla F, ~ \forall F\in\mathcal{C}^{2}(\mathbb{R}_{b}^{n},\mathbb{R})$, it follows that $X=\nabla^{L}_{b}H$ for some $H\in\mathcal{C}^{2}(\mathbb{R}_{b}^{n},\mathbb{R})$ if and only if $(B^{\star})^{-1}X= \nabla H$ for some $H\in\mathcal{C}^{2}(\mathbb{R}_{b}^{n},\mathbb{R})$. Due to the classical Poincar\'e lemma, the former equality holds true if and only if $\mathrm{D}((B^\star)^{-1}X)(\mathbf{x})=[\mathrm{D}((B^{\star})^{-1}X)(\mathbf{x})]^{\top}$, $\forall\mathbf{x}\in\mathbb{R}_{b}^{n}$ which is in turn equivalent to $(\mathrm{D}X(\mathbf{x}))^{\top}B^{-1}=(B^{\star})^{-1}\mathrm{D}X(\mathbf{x})$, $\forall \mathbf{x}\in\mathbb{R}_{b}^{n}$; taking into account that $B^{\top}=B^{\star}$.

Note that if $H$ satisfies $X=\nabla^{L}_{b}H$, then for each arbitrary fixed $\mathbf{x}\in\mathbb{R}^n$ we have that
\begin{align*}
H(\mathbf{x})- H(\mathbf{0})&=\int_{0}^{1}\dfrac{\mathrm{d}}{\mathrm{d}t}H(t\mathbf{x})\mathrm{d}t = \int_{0}^{1}\mathrm{d}H(t\mathbf{x})\cdot \mathbf{x}\mathrm{d}t = \int_{0}^{1} b( \nabla_{b}^{L} H(t\mathbf{x}) , \mathbf{x}) \mathrm{d}t \\
&=\int_{0}^{1} b( X(t\mathbf{x}) , \mathbf{x}) \mathrm{d}t.
\end{align*}

Given a vector field $X$ such that the relation \eqref{relunu} holds true, we check that the function $H(\mathrm{x})=\int_{0}^{1}b(X(t\mathbf{x}),\mathbf{x})\mathrm{d}t, ~ \forall \mathbf{x}\in\mathbb{R}_{b}^{n}$ is a solution of the equation  $X=\nabla^{L}_{b}H$.

Indeed, as $(b,B)$ is a geometric pair, $H$ can be equivalently written as $H(\mathbf{x})=\int_{0}^{1}\langle X(t\mathbf{x}),B^{-1}\mathbf{x}\rangle\mathrm{d}t, ~ \forall \mathbf{x}\in\mathbb{R}_{b}^{n}$. Thus, using the relation \eqref{relunu}, and the fact that $B^{\star}=B^{\top}$, we obtain successively the following equalities valid for all $\mathbf{x},\mathbf{v}\in\mathbb{R}^n$:

\begin{align*}
b(\nabla_{b}^{L}H(\mathbf{x}),\mathbf{v})&=\mathrm{d}H(\mathbf{x})\cdot \mathbf{v} = \int_{0}^{1}\left(\langle \mathrm{D}X(t\mathbf{x})\cdot t\mathbf{v} ,B^{-1}\mathbf{x}\rangle + \langle X(t\mathbf{x}) ,B^{-1}\mathbf{v} \rangle\right)\mathrm{d}t \\
&= \int_{0}^{1}\left(   t \mathbf{v}^{\top} (\mathrm{D}X(t\mathbf{x}))^{\top}B^{-1}\mathbf{x}    + \langle X(t\mathbf{x}) ,B^{-1}\mathbf{v} \rangle\right)\mathrm{d}t\\
&=\int_{0}^{1}\left(   t \mathbf{v}^{\top} (B^{\star})^{-1}\mathrm{D}X(t\mathbf{x})\mathbf{x}    + \langle X(t\mathbf{x}) ,B^{-1}\mathbf{v} \rangle\right)\mathrm{d}t\\
&=\int_{0}^{1}\left(   t \mathbf{v}^{\top} (B^{\top})^{-1}\mathrm{D}X(t\mathbf{x})\mathbf{x}    + \langle X(t\mathbf{x}) ,B^{-1}\mathbf{v} \rangle\right)\mathrm{d}t\\
&=\int_{0}^{1}\left(   t \mathbf{v}^{\top} (B^{-1})^{\top}\mathrm{D}X(t\mathbf{x})\mathbf{x}    + \langle X(t\mathbf{x}) ,B^{-1}\mathbf{v} \rangle\right)\mathrm{d}t\\
&=\int_{0}^{1}\left(\langle t B^{-1}\mathbf{v}, \mathrm{D}X(t\mathbf{x})\cdot \mathbf{x} \rangle + \langle X(t\mathbf{x}) ,B^{-1}\mathbf{v} \rangle\right)\mathrm{d}t\\
&=\int_{0}^{1}\left(\langle t\mathrm{D}X(t\mathbf{x})\cdot \mathbf{x}, B^{-1}\mathbf{v}\rangle + \langle X(t\mathbf{x}) ,B^{-1}\mathbf{v} \rangle\right)\mathrm{d}t\\
&=\int_{0}^{1}\langle t\mathrm{D}X(t\mathbf{x})\cdot \mathbf{x}+ X(t\mathbf{x}) , B^{-1}\mathbf{v}\rangle\mathrm{d}t \\
&= \langle \int_{0}^{1}(t\mathrm{D}X(t\mathbf{x})\cdot \mathbf{x}+ X(t\mathbf{x}))\mathrm{d}t , B^{-1}\mathbf{v}\rangle\\
&=\langle\int_{0}^{1}\dfrac{\mathrm{d}}{\mathrm{d}t}\left[t X(t\mathbf{x})\right]\mathrm{d}t,B^{-1}\mathbf{v}\rangle = \langle X(\mathbf{x}),B^{-1}\mathbf{v}\rangle = b(X(\mathbf{x}),\mathbf{v}),
\end{align*}
and hence we get $X=\nabla^{L}_{b}H$.

Now, if $H_1, H_2$ are two solutions of the equation $X=\nabla^{L}_{b}H$, then $\nabla H_1 (\mathbf{x}) =\nabla H_2 (\mathbf{x}), ~ \forall \mathbf{x}\in\mathbb{R}_{b}^{n}$, and hence $H_1(\mathbf{x}) - H_2 (\mathbf{x}) = constant, ~\forall\mathbf{x}\in\mathbb{R}_{b}^{n} $. 

\item[(ii)] The proof follows mimetically the one of item $(i)$.

\item[(iii)] The proof follows directly from $(i)$ taking into account that $B^{\star}=B$ if $b$ is symmetric.

\item[(iv)] The proof follows directly from $(ii)$ taking into account that $B^{\star}= - B$ if $b$ is skew--symmetric.
\end{enumerate}
\end{proof}

Next, we provide a geometric meaning of the solvability conditions, \eqref{relunu}, \eqref{reldoi}, given by the Poincar\'e lemma. In order to do that we need to introduce first some terminology.

\begin{definition}
Let $b$ be a geometric structure on $\mathbb{R}^{n}$, and let $(b,B)$ be the associated geometric pair. A linear map $A:\mathbb{R}_{b}^{n}\rightarrow \mathbb{R}_{b}^{n}$ will be called \textbf{left $(b,B)-$symmetric} if
\begin{equation*}
b(A\mathbf{u},\mathbf{v})=b(B^{\star}B^{-1}\mathbf{u},A\mathbf{v}), ~ \forall\mathbf{u},\mathbf{v}\in\mathbb{R}_{b}^{n},
\end{equation*}
respectively, \textbf{right $(b,B)-$symmetric} if
\begin{equation*}
b(A\mathbf{u},B(B^{\star})^{-1}\mathbf{v})=b(\mathbf{u},A\mathbf{v}), ~ \forall\mathbf{u},\mathbf{v}\in\mathbb{R}_{b}^{n}.
\end{equation*}
\end{definition}

\begin{remark}
The condition \eqref{relunu} means that $DX(\mathbf{x})$ is left $(b,B)-$symmetric for all $\mathbf{x}\in\mathbb{R}_{b}^{n}$, whereas the condition \eqref{reldoi} means that $DX(\mathbf{x})$ is right $(b,B)-$symmetric for all $\mathbf{x}\in\mathbb{R}_{b}^{n}$.
\end{remark}

As the left/right $(b,B)-$symmetry represents actually the solvability condition provided by the Poincar\'e lemma, we shall analyze further this feature by considering the set of all linear maps which share this important property. In order to do that, let us denote by
$$
\operatorname{Sym}_{n}^{L}(b,B):=\{A\in\operatorname{End}(\mathbb{R}_{b}^n): ~ b(A\mathbf{u},\mathbf{v})=b(B^{\star}B^{-1}\mathbf{u},A\mathbf{v}), ~ \forall\mathbf{u},\mathbf{v}\in\mathbb{R}_{b}^{n}\},
$$
the set of left $(b,B)-$symmetric linear maps on $\mathbb{R}_{b}^{n}$. Analogously, we denote by
$$
\operatorname{Sym}_{n}^{R}(b,B):=\{A\in\operatorname{End}(\mathbb{R}_{b}^n): ~ b(A\mathbf{u},B(B^{\star})^{-1}\mathbf{v})=b(\mathbf{u},A\mathbf{v}), ~ \forall\mathbf{u},\mathbf{v}\in\mathbb{R}_{b}^{n}\},
$$
the set of right $(b,B)-$symmetric linear maps on $\mathbb{R}_{b}^{n}$.

Note that both sets, $\operatorname{Sym}_{n}^{L}(b,B)$, $\operatorname{Sym}_{n}^{R}(b,B)$, are linear subspaces of the real vector space $\operatorname{End}(\mathbb{R}_{b}^n)$, consisting of all $\mathbb{R}-$linear maps $A: \mathbb{R}_{b}^{n}\rightarrow \mathbb{R}_{b}^{n}$. As $\operatorname{End}(\mathbb{R}_{b}^n)$ together with the commutator $[\cdot,\cdot]$ (given by $[A,A']:= AA'- A'A, ~ \forall A,A'\in \operatorname{End}(\mathbb{R}_{b}^n)$), becomes a Lie algebra, it is natural to analyze some related properties with respect to the subspaces $\operatorname{Sym}_{n}^{L}(b,B)$, $\operatorname{Sym}_{n}^{R}(b,B)$.

In order to do that, we introduce two more brackets on $\operatorname{End}(\mathbb{R}_{b}^n)$, induced by the geometric pair $(b,B)$. More precisely, one can easily check that each of the following brackets
\begin{align*}
[A,A']_{(b,B)}^{L}:&=AB(B^{\star})^{-1}A' - A'B(B^{\star})^{-1}A, ~ \forall A,A' \in \operatorname{End}(\mathbb{R}_{b}^n),\\
[A,A']_{(b,B)}^{R}:&=AB^{\star}B^{-1}A' - A'B^{\star}B^{-1}A, ~ \forall A,A' \in \operatorname{End}(\mathbb{R}_{b}^n),
\end{align*}
defines a Lie algebra structure on the vector space $\operatorname{End}(\mathbb{R}_{b}^n)$.

Some straightforward computations lead to interesting compatibility relations between the above defined brackets and the commutator bracket, synthesized in the following result.

\begin{proposition} Let $b$ be a geometric structure on $\mathbb{R}^{n}$ and let $(b,B)$ be the associated geometric pair. Then the following assertions hold true.
\begin{enumerate}
\item[(i)] For every $A,A' \in \operatorname{Sym}_{n}^{L}(b,B)$, we have
\begin{equation}\label{liestg}
b([A,A']_{(b,B)}^{L}\mathbf{u},\mathbf{v})= b(B^{\star}B^{-1}\mathbf{u},-[A,A']\mathbf{v}), ~ \forall \mathbf{u},\mathbf{v}\in\mathbb{R}_{b}^{n}.
\end{equation}
\item[(ii)] For every $A,A' \in \operatorname{Sym}_{n}^{R}(b,B)$, we have
\begin{equation}\label{liedr}
b(-[A,A']\mathbf{u},B(B^{\star})^{-1}\mathbf{v})=b(\mathbf{u},[A,A']_{(b,B)}^{R}\mathbf{v}), ~ \forall \mathbf{u},\mathbf{v}\in\mathbb{R}_{b}^{n}.
\end{equation}
\end{enumerate}
\end{proposition}

Next result translates the above properties in the special case when the geometric structure $b$ is symmetric or skew--symmetric. Moreover, we also provide geometric interpretations of the solvability conditions \eqref{reltrei}, \eqref{relpatru}.

\begin{remark} 
\begin{enumerate}
\item[(i)] If $b$ is symmetric then $B^{\star}=B$ and consequently 
\begin{align*}
\operatorname{Sym}_{n}^{L}(b,B)&=\operatorname{Sym}_{n}^{R}(b,B)\\
&=\operatorname{Sym}_{n}(b)=\{A\in\operatorname{End}(\mathbb{R}_{b}^n): ~ b(A\mathbf{u},\mathbf{v})=b(\mathbf{u},A\mathbf{v}), ~ \forall\mathbf{u},\mathbf{v}\in\mathbb{R}_{b}^{n}\},
\end{align*}
where $\operatorname{Sym}_{n}(b)$ denotes the space of linear self--adjoint operators on $\mathbb{R}_{b}^n$ (or the space of linear operators symmetric with respect to $b$).

Moreover, as $[A,A']_{(b,B)}^{L}=[A,A']_{(b,B)}^{R}= [A,A'], ~\forall A,A' \in \operatorname{End}(\mathbb{R}_{b}^n)$,  
the relations \eqref{liestg}, \eqref{liedr} are identical, and hence for every $A,A' \in \operatorname{Sym}_{n}(b)$ we have that
$$
b([A,A']\mathbf{u},\mathbf{v})= - b(\mathbf{u},[A,A']\mathbf{v}), ~ \forall \mathbf{u},\mathbf{v}\in\mathbb{R}_{b}^{n}.
$$

Returning to the solvability condition \eqref{reltrei}, this becomes equivalent to $$DX(\mathbf{x})\in\operatorname{Sym}_{n}(b),~ \forall\mathbf{x}\in\mathbb{R}_{b}^{n},$$
which in the case when $b$ is positive definite (i.e. $b$ is an inner product), is precisely the classical characterization of $X$ as being a gradient vector field with respect to the inner product $b$.

\item[(ii)] If $n$ is even and $b$ is skew-symmetric (i.e. $b$ is a symplectic form) then $B^{\star}= - B$ and consequently 
\begin{align*}
\operatorname{Sym}_{n}^{L}(b,B)&=\operatorname{Sym}_{n}^{R}(b,B)\\
&=\operatorname{Skew}_{n}(b)=\{A\in\operatorname{End}(\mathbb{R}_{b}^n): ~ b(A\mathbf{u},\mathbf{v})= - b(\mathbf{u},A\mathbf{v}), ~ \forall\mathbf{u},\mathbf{v}\in\mathbb{R}_{b}^{n}\},
\end{align*}
where $\operatorname{Skew}_{n}(b)$ denotes the space of linear skew--adjoint operators on $\mathbb{R}_{b}^n$ (or the space of linear operators skew-symmetric with respect to $b$).

Moreover, as $[A,A']_{(b,B)}^{L}=[A,A']_{(b,B)}^{R}= - [A,A'], ~\forall A,A' \in \operatorname{End}(\mathbb{R}_{b}^n)$,  
the relations \eqref{liestg}, \eqref{liedr} are identical, and hence for every $A,A' \in \operatorname{Skew}_{n}(b)$ we have that
$$
b([A,A']\mathbf{u},\mathbf{v})=- b(\mathbf{u},[A,A']\mathbf{v}), ~ \forall \mathbf{u},\mathbf{v}\in\mathbb{R}_{b}^{n}.
$$
Consequently, $[A,A']\in\operatorname{Skew}_{n}(b), ~ \forall A,A' \in \operatorname{Skew}_{n}(b)$, and so $\operatorname{Skew}_{n}(b)$ becomes a Lie subalgebra of $\operatorname{End}(\mathbb{R}_{b}^n)$. More precisely, as $b$ is a symplectic form on $\mathbb{R}^n$, it follows that $\operatorname{Skew}_{n}(b)=\mathfrak{sp}(b)$, where $\mathfrak{sp}(b)$ denotes the Lie algebra of the symplectic group generated by $b$, i.e. $\operatorname{Sp}(b):=\{A\in\operatorname{Aut}(\mathbb{R}_{b}^n): ~ b(A\mathbf{u},A\mathbf{v})=b(\mathbf{u},\mathbf{v}), ~\forall \mathbf{u},\mathbf{v}\in\mathbb{R}_{b}^{n}\}$.

Returning to the solvability condition \eqref{relpatru}, this becomes equivalent to $$DX(\mathbf{x})\in\mathfrak{sp}(b),~ \forall\mathbf{x}\in\mathbb{R}_{b}^{n},$$
which is precisely the classical characterization of $X$ as being a Hamiltonian vector field with respect to the symplectic form $b$.
\end{enumerate}
\end{remark}

Let us present now some dynamical properties of the gradient--like vector fields $\nabla^{R}_{b}F$, $\nabla^{R}_{b}F$, where $F:\mathbb{R}^{n}_{b}\rightarrow \mathbb{R}$ is an arbitrary given function of class $\mathcal{C}^1$. In the following, $\mathcal{L}_X$ stands for the Lie derivative along the vector field $X$.

\begin{proposition}\label{dyna}
Let $b$ be a geometric structure on $\mathbb{R}^{n}$, and let $(b,B)$ be the associated geometric pair. Then for every $F\in\mathcal{C}^{1}(\mathbb{R}^{n}_{b},\mathbb{R})$ the following assertions hold true.
\begin{enumerate}
\item[(i)] 
\begin{align*}\left(\mathcal{L}_{\nabla^{R}_{b}F} F\right)(\mathbf{x})&=\left(\mathcal{L}_{\nabla^{L}_{b}F} F\right)(\mathbf{x})= b(\nabla^{L}_{b}F(\mathbf{x}),\nabla^{L}_{b}F(\mathbf{x}))\\
&=b(\nabla^{R}_{b}F(\mathbf{x}),\nabla^{R}_{b}F(\mathbf{x}))=(\nabla F(\mathbf{x}))^{\top}B\nabla F(\mathbf{x}), ~\forall \mathbf{x}\in\mathbb{R}^{n}_{b},
\end{align*}
\item[(ii)] If $b$ is skew--symmetric then $F$ is a first integral for the vector field $\nabla_{b}F$.
\item[(iii)] If $b$ is symmetric and positive (negative) definite then $\mathcal{L}_{\nabla_{b}F} F\geq 0$ ($\mathcal{L}_{\nabla_{b}F} F\leq  0$) with equality on the set of critical points of $F$.
\end{enumerate}
\end{proposition}
\begin{proof}
\begin{enumerate}
\item[(i)] Using the definition of the Lie derivative and the fact that $\nabla^{R}_{b}F =B\nabla F$, we obtain successively
\begin{align}\label{lie1}
\begin{split}
\left(\mathcal{L}_{\nabla^{R}_{b}F} F\right)(\mathbf{x})&=\langle\nabla^{R}_{b}F(\mathbf{x}),\nabla F(\mathbf{x})\rangle =\langle B\nabla F(\mathbf{x}),\nabla F(\mathbf{x})\rangle\\
& = b(B\nabla F(\mathbf{x}),B\nabla F(\mathbf{x}))= b(\nabla^{R}_{b}F(\mathbf{x}),\nabla^{R}_{b}F(\mathbf{x})), ~\forall \mathbf{x}\in\mathbb{R}^{n}_{b}.
\end{split}
\end{align}
Similarly, using the relation $\nabla^{L}_{b}F =B^{\star}\nabla F$ it follows that
\begin{align}\label{lie2}
\begin{split}
\left(\mathcal{L}_{\nabla^{L}_{b}F} F\right)(\mathbf{x})&=\langle\nabla^{L}_{b}F(\mathbf{x}),\nabla F(\mathbf{x})\rangle =\langle B^{\star}\nabla F(\mathbf{x}),\nabla F(\mathbf{x})\rangle = \langle \nabla F(\mathbf{x}),B^{\star}\nabla F(\mathbf{x})\rangle\\
& = b(B^{\star}\nabla F(\mathbf{x}),B^{\star}\nabla F(\mathbf{x}))= b(\nabla^{L}_{b}F(\mathbf{x}),\nabla^{L}_{b}F(\mathbf{x})), ~\forall \mathbf{x}\in\mathbb{R}^{n}_{b}.
\end{split}
\end{align}
As $B^{\star}=B^{\top}$ we get that $$\langle B\nabla F(\mathbf{x}),\nabla F(\mathbf{x})\rangle=\langle B^{\star}\nabla F(\mathbf{x}),\nabla F(\mathbf{x})\rangle =(\nabla F(\mathbf{x}))^{\top}B\nabla F(\mathbf{x}), ~\forall \mathbf{x}\in\mathbb{R}^{n}_{b},$$ and hence the relations \eqref{lie1} and \eqref{lie2} are identical.
\item[(ii)] The proof follows from $(i)$ together with the equality $b(\mathbf{x},\mathbf{x})=0, ~ \forall \mathbf{x}\in\mathbb{R}^{n}_{b}$.
\item[(iii)] The proof follows from $(i)$ taking into account that the vector fields $\nabla_{b}F$ and $\nabla F$, have the same equilibrium points.
\end{enumerate}
\end{proof}

In the following, we express some of the classical vector fields naturally associated to three of the main canonical geometries of $\mathbb{R}^n$, as gradient--like vector fields associated to specific geometric structures on $\mathbb{R}^n$.

\begin{example}\label{expl}
\begin{enumerate}
\item[(i)] If $b(\mathbf{x},\mathbf{y})=\langle \mathbf{x},\mathbf{y} \rangle, ~\forall \mathbf{x},\mathbf{y}\in\mathbb{R}^{n}$, then $\mathbb{R}^{n}_{b}$ is the canonical Euclidean $n-$dimensional vector space. Thus, in this case we have that $\nabla_{b}F=\nabla F, ~ \forall F\in\mathcal{C}^{1}(\mathbb{R}^{n}_{b},\mathbb{R})$.
\item[(ii)] If $n=2m$ and $b(\mathbf{x},\mathbf{y})=\omega_{0}(\mathbf{x},\mathbf{y}):=\langle \mathbf{x},\mathbb{J}_{n} \mathbf{y} \rangle, ~\forall \mathbf{x},\mathbf{y}\in\mathbb{R}^{n}$, where $\mathbb{J}_{n}:= \left[ {\begin{array}{*{20}c}
   O_m & I_m \\
   - I_m & O_m \\
\end{array}} \right]$, then $\mathbb{R}^{n}_{b}$ is the canonical symplectic $n-$dimensional vector space. Thus, in this case we have that $\nabla_{b}F =X_{F}, ~\forall F\in\mathcal{C}^{1}(\mathbb{R}^{n}_{b},\mathbb{R})$, where $X_{F}=\mathbb{J}_n\nabla F$ is the Hamiltonian vector field of $F$ with respect to the canonical symplectic form $\omega_{0}$. 
\item[(iii)] If $b(\mathbf{x},\mathbf{y})=\langle \mathbf{x},\mathbb{E}_{n-1,1} \mathbf{y} \rangle, ~\forall \mathbf{x},\mathbf{y}\in\mathbb{R}^{n}$, where $\mathbb{E}_{n-1,1}:= \left[ {\begin{array}{*{20}c}
   I_{n-1} & O_{n-1,1} \\
   O_{1,n-1} & -1 \\
\end{array}} \right]$, then $\mathbb{R}^{n}_{b}$ is the canonical Minkowski $n-$dimensional vector space $\mathbb{R}^{n-1,1}$. Thus, in this case we have that $\nabla_{b}F =\nabla_{n-1,1}F, ~\forall F\in\mathcal{C}^{1}(\mathbb{R}^{n}_{b},\mathbb{R})$, where $\nabla_{n-1,1}F=\mathbb{E}_{n-1,1}\nabla F$ is 
the Minkowski gradient vector field of $F$ with respect to the canonical $n-$dimensional Minkowski inner product on $\mathbb{R}^n$.
\end{enumerate} 
\end{example}

\section{A global geometric decomposition of vector fields of class $\mathcal{C}^{1}$ on $\mathbb{R}^{n}$}

In this section we present the main result of this article, which provides a global geometric decomposition of a general continuously differentiable vector field on $\mathbb{R}^n$. More precisely, given a vector field of class $\mathcal{C}^{1}$ on $\mathbb{R}^{n}$, and a geometric structure on $\mathbb{R}^n$, we provide a unique global decomposition of the vector field as a sum of a left (right) gradient--like vector field (naturally associated to the geometric structure) with potential function vanishing at the origin, and a vector field which is left (right) orthogonal to the identity, with respect to the geometric structure. This decomposition is a geometric generalization of the main result of \cite{presnov}.

\begin{theorem}\label{MTHM}
Let $X\in\mathfrak{X}(\mathbb{R}^{n})$ be a vector field of class $\mathcal{C}^{1}$ on $\mathbb{R}^{n}$. Let $b$ be a geometric structure on $\mathbb{R}^{n}$. Then the following assertions hold.
\begin{enumerate}
\item[(i)] There exists a \textbf{unique global decomposition} of $X$ given by
\begin{equation}\label{split1}
X(\mathbf{x})=\nabla^{R}_{b}H(\mathbf{x}) + \mathbf{u}(\mathbf{x}), ~ \forall \mathbf{x}\in\mathbb{R}^{n}_{b},
\end{equation}
such that $H(\mathbf{0})=0$ and $b(\mathbf{x},\mathbf{u}(\mathbf{x}))=0, ~\forall \mathbf{x}\in\mathbb{R}^{n}_{b}$.
\item[(ii)] There exists a \textbf{unique global decomposition} of $X$ given by
\begin{equation}\label{split2}
X(\mathbf{x})=\nabla^{L}_{b}H^{\star}(\mathbf{x}) + \mathbf{u}^{\star}(\mathbf{x}), ~ \forall \mathbf{x}\in\mathbb{R}^{n}_{b},
\end{equation}
such that $H^{\star}(\mathbf{0})=0$ and $b(\mathbf{u}^{\star}(\mathbf{x}),\mathbf{x})=0, ~\forall \mathbf{x}\in\mathbb{R}^{n}_{b}$.
\item[(iii)] The relation between $H$ and $H^{\star}$ is given by
\begin{equation}\label{hhs}
H^{\star}(\mathbf{x})= H(\mathbf{x}) + 2 \int_{0}^{1}\mathcal{A}_{b}(X(t\mathbf{x}),\mathbf{x})\mathrm{d}t, ~ \forall \mathbf{x}\in\mathbb{R}^{n}_{b},
\end{equation}
where $\mathcal{A}_{b}(\mathbf{x},\mathbf{y}):=\dfrac{1}{2}\left( b(\mathbf{x},\mathbf{y})-b(\mathbf{y},\mathbf{x})\right), ~\forall \mathbf{x},\mathbf{y}\in\mathbb{R}^{n}_{b}$, is the skew--symmetric part of the bilinear form $b$. 
\item[(iv)] If $b$ is symmetric, then $H^{\star}=H$ and $\mathbf{u}^{\star}=\mathbf{u}$, while if $b$ is skew--symmetric, then $H^{\star}= - H$ and $\mathbf{u}^{\star}=\mathbf{u}$.
\item[(v)] If $b$ is symmetric or skew--symmetric, there exists a \textbf{unique global decomposition} of $X$ given by
\begin{equation}\label{decvf}
X(\mathbf{x})=\nabla_{b}H(\mathbf{x}) + \mathbf{u}(\mathbf{x}), ~ \forall \mathbf{x}\in\mathbb{R}^{n}_{b},
\end{equation}
such that $H(\mathbf{0})=0$ and $b(\mathbf{u}(\mathbf{x}),\mathbf{x})=b(\mathbf{x},\mathbf{u}(\mathbf{x}))=0, ~\forall \mathbf{x}\in\mathbb{R}^{n}_{b}$.

When $b$ is skew--symmetric (i.e. symplectic form), $\nabla_{b}H$ coincides with $X_{H}$, the Hamiltonian vector field of $H$ with respect to the symplectic form $b$.
\end{enumerate}
\end{theorem}
\begin{proof} The proof is based on the construction of solutions provided by the Poincar\'e lemma associated to the geometric structure $b$, i.e. the Theorem \ref{Poincare}.
\begin{enumerate}
\item[(i)] Let us consider the function $\sigma : \mathbb{R}^{n}_{b}\rightarrow \mathbb{R}$ defined by $\sigma(\mathbf{x})=b(\mathbf{x},X(\mathbf{x})), ~ \forall \mathbf{x}\in\mathbb{R}^{n}_{b}$. As $\sigma(t\mathbf{x})=t b(\mathbf{x},X(t\mathbf{x})), ~ \forall \mathbf{x}\in\mathbb{R}^{n}_{b}, \forall t\in\mathbb{R}$, it is well defined the function $H:\mathbb{R}^{n}_{b}\rightarrow \mathbb{R}$ given by
\begin{equation}\label{h1}
H(\mathbf{x})=\int_{0}^{1}\dfrac{1}{t}\sigma(t\mathbf{x})\mathrm{d}t, ~~\forall \mathbf{x}\in\mathbb{R}^{n}_{b}.
\end{equation}
Since $\sigma(\mathbf{0})=0$, we get that $H(\mathbf{0})=0$.

Using the definition of $H$, we obtain successively the following equalities valid for all $\mathbf{x}=(x_1,\dots, x_n)\in\mathbb{R}^{n}_{b}$:
\begin{align*}
b(\mathbf{x},\nabla^{R}_{b}H(\mathbf{x}))&=\mathrm{d}H(\mathbf{x})\cdot\mathbf{x}=\langle\mathbf{x},\nabla H(\mathbf{x})\rangle=\sum_{i=1}^{n}x_i \cdot \dfrac{\partial}{\partial x_i}\int_{0}^{1}\dfrac{1}{t}\sigma(t\mathbf{x})\mathrm{d}t\\
&=\sum_{i=1}^{n}x_i \cdot \int_{0}^{1}\dfrac{\partial\sigma}{\partial x_i}(t\mathbf{x})\mathrm{d}t =\int_{0}^{1}\dfrac{\mathrm{d}}{\mathrm{d}t}\sigma(t\mathbf{x})\mathrm{d}t =\sigma(\mathbf{x})=b(\mathbf{x},X(\mathbf{x})).
\end{align*}
Consequently, from the above relations and the bilinearity of $b$ it follows that $$b(\mathbf{x},X(\mathbf{x})-\nabla^{R}_{b} H(\mathbf{x}))=0, ~ \forall \mathbf{x}\in\mathbb{R}^{n}_{b}.$$
Denoting $\mathbf{u}(\mathbf{x}):=X(\mathbf{x})- \nabla^{R}_{b} H(\mathbf{x}), ~~ \forall \mathbf{x}\in\mathbb{R}^{n}_{b}$, we get that $b(\mathbf{x},\mathbf{u}(\mathbf{x}))=0, ~\forall \mathbf{x}\in\mathbb{R}^{n}_{b}$, and hence the splitting \eqref{split1}.

In order to prove the uniqueness of this splitting, suppose that given $X\in\mathfrak{X}(\mathbb{R}^{n}_{b})$, there exist two pairs $(H_1, \mathbf{u}_1)$ and $(H_2, \mathbf{u}_2)$ with $H_1 (\mathbf{0})=H_2 (\mathbf{0})=0$ and $b(\mathbf{x},\mathbf{u}_1(\mathbf{x}))=b(\mathbf{x},\mathbf{u}_2(\mathbf{x}))=0, ~\forall \mathbf{x}\in\mathbb{R}^{n}_{b}$, such that
\begin{equation}\label{e1}
X(\mathbf{x})=\nabla^{R}_{b} H_1 (\mathbf{x}) + \mathbf{u}_1 (\mathbf{x}), ~ \forall \mathbf{x}\in\mathbb{R}^{n}_{b},
\end{equation}
and 
\begin{equation}\label{e2}
X(\mathbf{x})=\nabla^{R}_{b} H_2 (\mathbf{x}) + \mathbf{u}_2 (\mathbf{x}), ~ \forall \mathbf{x}\in\mathbb{R}^{n}_{b}.
\end{equation}
Using the relations $b(\mathbf{x},\mathbf{u}_1(\mathbf{x}))=b(\mathbf{x},\mathbf{u}_2(\mathbf{x}))=0, ~\forall \mathbf{x}\in\mathbb{R}^{n}_{b}$ we deduce that
\begin{equation}\label{e3}
b(\mathbf{x},\mathbf{u}_1(\mathbf{x})-\mathbf{u}_2(\mathbf{x}))=0, ~ \forall \mathbf{x}\in\mathbb{R}^{n}_{b}.
\end{equation}
Now, from the equalities \eqref{e1}, \eqref{e2} one gets that 
\begin{equation*}
\mathbf{u}_1(\mathbf{x})-\mathbf{u}_2(\mathbf{x})= \nabla^{R}_{b} (H_2 -H_1)(\mathbf{x}),~ \forall \mathbf{x}\in\mathbb{R}^{n}_{b},
\end{equation*}
and consequently, the relation \eqref{e3} reads as follows
\begin{equation*}
b(\mathbf{x},\nabla^{R}_{b} (H_2 -H_1)(\mathbf{x}))=0, ~ \forall \mathbf{x}\in\mathbb{R}^{n}_{b}.
\end{equation*}
As $b(\mathbf{v},\nabla^{R}_{b} (H_2 -H_1)(\mathbf{x}))=\mathrm{d}(H_2 -H_1)(\mathbf{x})\cdot\mathbf{v}, ~ \forall \mathbf{x},\mathbf{v}\in\mathbb{R}^{n}_{b}$, the above relation becomes 
\begin{equation*}
\langle\mathbf{x},\nabla (H_2 -H_1)(\mathbf{x})\rangle=0, ~ \forall \mathbf{x}\in\mathbb{R}^{n}_{b},
\end{equation*}
which means that $H_2 - H_1$ is a \textit{global} first integral of the vector field $\mathbb{X}(\mathbf{x}):=\mathbf{x}, ~\forall\mathbf{x}\in\mathbb{R}^{n}_{b}$. Thus, the function $(H_2 - H_1) :\mathbb{R}^{n}_{b}\rightarrow \mathbb{R}$ must be constant. Since by hypothesis, $H_1 (\mathbf{0})=H_2 (\mathbf{0})=0$, we get $H_1 =H_2$. Using the relations \eqref{e1}, \eqref{e2}, it follows that $\mathbf{u}_1=\mathbf{u}_2$, and hence we obtain the uniqueness of the splitting \eqref{split1}.

\item[(ii)] The proof follows step by step the one of item $(i)$. More precisely, in this case we start with the function $\sigma^{\star} : \mathbb{R}^{n}_{b}\rightarrow \mathbb{R}$ defined by $\sigma^{\star}(\mathbf{x})=b(X(\mathbf{x}),\mathbf{x}), ~ \forall \mathbf{x}\in\mathbb{R}^{n}_{b}$. Associated to $\sigma^{\star}$ we define the function $H^{\star}:\mathbb{R}^{n}_{b}\rightarrow \mathbb{R}$ given by
\begin{equation}\label{h2}
H^{\star}(\mathbf{x})=\int_{0}^{1}\dfrac{1}{t}\sigma^{\star}(t\mathbf{x})\mathrm{d}t, ~~\forall \mathbf{x}\in\mathbb{R}^{n}_{b}.
\end{equation}
As in the previous case, we have that $H^{\star}(\mathbf{0})=0$.

Using the definition of $H^{\star}$, we obtain successively the following equalities valid for all $\mathbf{x}=(x_1,\dots, x_n)\in\mathbb{R}^{n}_{b}$:
\begin{align*}
b(\nabla^{L}_{b} H^{\star}(\mathbf{x}),\mathbf{x})&=\mathrm{d}H^{\star}(\mathbf{x})\cdot\mathbf{x}=\langle\mathbf{x},\nabla H^{\star}(\mathbf{x})\rangle =\sum_{i=1}^{n}x_i \cdot \dfrac{\partial}{\partial x_i}\int_{0}^{1}\dfrac{1}{t}\sigma^{\star}(t\mathbf{x})\mathrm{d}t \\
&=\sum_{i=1}^{n}x_i \cdot \int_{0}^{1}\dfrac{\partial\sigma^{\star}}{\partial x_i}(t\mathbf{x})\mathrm{d}t =\int_{0}^{1}\dfrac{\mathrm{d}}{\mathrm{d}t}\sigma^{\star}(t\mathbf{x})\mathrm{d}t =\sigma^{\star}(\mathbf{x})=b(X(\mathbf{x}),\mathbf{x}).
\end{align*}
Consequently, from the above relations and the bilinearity of $b$ it follows that $$b(X(\mathbf{x})-\nabla^{L}_{b} H^{\star}(\mathbf{x}),\mathbf{x})=0, ~ \forall \mathbf{x}\in\mathbb{R}^{n}_{b}.$$
Denoting $\mathbf{u}^{\star}(\mathbf{x}):=X(\mathbf{x})- \nabla^{L}_{b} H^{\star}(\mathbf{x}), ~~ \forall \mathbf{x}\in\mathbb{R}^{n}_{b}$, we get that $b(\mathbf{u}^{\star}(\mathbf{x}),\mathbf{x})=0, ~\forall \mathbf{x}\in\mathbb{R}^{n}_{b}$, and hence the splitting \eqref{split2}.

In order to prove the uniqueness of this splitting, suppose that given $X\in\mathfrak{X}(\mathbb{R}^{n}_{b})$, there exist two pairs $(H_1^{\star}, \mathbf{u}_1^{\star})$ and $(H_2^{\star}, \mathbf{u}_2^{\star})$ with $H_1^{\star} (\mathbf{0})=H_2^{\star} (\mathbf{0})=0$ and $b(\mathbf{u}_1^{\star}(\mathbf{x}),\mathbf{x})=b(\mathbf{u}_2^{\star}(\mathbf{x}),\mathbf{x})=0, ~\forall \mathbf{x}\in\mathbb{R}^{n}_{b}$, such that
\begin{equation}\label{e1s}
X(\mathbf{x})=\nabla^{L}_{b} H_1^{\star} (\mathbf{x}) + \mathbf{u}_1^{\star} (\mathbf{x}), ~ \forall \mathbf{x}\in\mathbb{R}^{n}_{b},
\end{equation}
and 
\begin{equation}\label{e2s}
X(\mathbf{x})=\nabla^{L}_{b} H_2^{\star} (\mathbf{x}) + \mathbf{u}_2^{\star} (\mathbf{x}), ~ \forall \mathbf{x}\in\mathbb{R}^{n}_{b}.
\end{equation}
Using the relations $b(\mathbf{u}_1^{\star}(\mathbf{x}),\mathbf{x})=b(\mathbf{u}_2^{\star}(\mathbf{x}),\mathbf{x})=0, ~\forall \mathbf{x}\in\mathbb{R}^{n}_{b}$ we deduce that
\begin{equation}\label{e3s}
b(\mathbf{u}_1^{\star}(\mathbf{x})-\mathbf{u}_2^{\star}(\mathbf{x}),\mathbf{x})=0, ~ \forall \mathbf{x}\in\mathbb{R}^{n}_{b}.
\end{equation}
Now, from the equalities \eqref{e1s}, \eqref{e2s} one gets that 
\begin{equation*}
\mathbf{u}_1^{\star}(\mathbf{x})-\mathbf{u}_2^{\star}(\mathbf{x})= \nabla^{L}_{b} (H_2^{\star} -H_1^{\star})(\mathbf{x}),~ \forall \mathbf{x}\in\mathbb{R}^{n}_{b},
\end{equation*}
and consequently, the relation \eqref{e3s} reads as follows
\begin{equation*}
b(\nabla^{L}_{b} (H_2^{\star} -H_1^{\star})(\mathbf{x}),\mathbf{x})=0, ~ \forall \mathbf{x}\in\mathbb{R}^{n}_{b}.
\end{equation*}
As $b(\nabla^{L}_{b} (H_2^{\star} -H_1^{\star})(\mathbf{x}),\mathbf{v})=\mathrm{d}(H_2^{\star} -H_1^{\star})(\mathbf{x})\cdot\mathbf{v}, ~ \forall \mathbf{x},\mathbf{v}\in\mathbb{R}^{n}_{b}$, the above relation becomes
\begin{equation*}
\langle\nabla (H_2^{\star} -H_1^{\star})(\mathbf{x}),\mathbf{x}\rangle=0, ~ \forall \mathbf{x}\in\mathbb{R}^{n}_{b},
\end{equation*}
which means that $H_2^{\star} - H_1^{\star}$ is a \textit{global} first integral of the vector field $\mathbb{X}(\mathbf{x})=\mathbf{x}, ~\forall\mathbf{x}\in\mathbb{R}_{b}^{n}$. Thus, the function $(H_2^{\star} - H_1^{\star}) :\mathbb{R}^{n}_{b}\rightarrow \mathbb{R}$ must be constant. Since by hypothesis, $H_1^{\star} (\mathbf{0})=H_2^{\star} (\mathbf{0})=0$, we get $H_1^{\star} =H_2^{\star}$. Using the relations \eqref{e1s}, \eqref{e2s}, it follows that $\mathbf{u}_1^{\star}=\mathbf{u}_2^{\star}$, and hence we obtain the uniqueness of the splitting \eqref{split2}.

\item[(iii)] Using the definitions of $H$ and $H^{\star}$ given by the relations \eqref{h1} and \eqref{h2}, we obtain successively 
\begin{align*}
H^{\star}(\mathbf{x})-H(\mathbf{x})&=\int_{0}^{1}\dfrac{1}{t}\left( b(X(t\mathbf{x}),t\mathbf{x})-b(t\mathbf{x},X(t\mathbf{x}))\right)\mathrm{d}t\\
&=2\int_{0}^{1}\dfrac{1}{t}\mathcal{A}_{b}(X(t\mathbf{x}),t\mathbf{x})\mathrm{d}t = 2\int_{0}^{1}\mathcal{A}_{b}(X(t\mathbf{x}),\mathbf{x})\mathrm{d}t, ~\forall \mathbf{x}\in\mathbb{R}^{n}_{b}.
\end{align*}

\item[(iv)] We split the proof in two cases, regarding the symmetry properties of $b$.
\begin{enumerate}
\item[(1)] If $b$ is symmetric, then $\mathcal{A}_{b}=0$ and hence from \eqref{hhs} it follows that $H^{\star}=H$. As $\nabla_{b}^{R} = \nabla_{b}^{L}$ and $H^{\star}=H$, we obtain that $\nabla_{b}^{R}H = \nabla_{b}^{L}H^{\star}$ and consequently from $(i)$ and $(ii)$ we get that $\mathbf{u}^{\star}=\mathbf{u}$.

\item[(2)] If $b$ is skew--symmetric, then $\mathcal{A}_{b}=b$ and hence from \eqref{hhs} it follows that $H^{\star}= - H$. As $\nabla_{b}^{R} = - \nabla_{b}^{L}$ and $H^{\star}= -H$, we obtain that $\nabla_{b}^{R}H = \nabla_{b}^{L}H^{\star}$ and consequently from $(i)$ and $(ii)$ we get that $\mathbf{u}^{\star}=\mathbf{u}$.
\end{enumerate}

\item[(v)] The proof follows directly from $(i)$ or $(ii)$.
\end{enumerate}
\end{proof}

Next result presents a conservative property of two vector fields naturally associated to the vector fields $\mathbf{u},\mathbf{u}^{\star}\in\mathfrak{X}(\mathbb{R}^{n}_{b})$ introduced in Theorem \ref{MTHM}. Equivalently, we give a dynamical interpretation of the relations $b(\mathbf{x},\mathbf{u}(\mathbf{x}))=0,~ b(\mathbf{u}^{\star}(\mathbf{x}),\mathbf{x})=0, ~\forall \mathbf{x}\in\mathbb{R}^{n}_{b}$.

\begin{proposition}\label{pr11}
Let $b$ be a geometric structure on $\mathbb{R}^n$, and let $(b,B)$ be the associated geometric pair. Then, the function $N:\mathbb{R}^{n}_{b}\rightarrow \mathbb{R}$, $N(\mathbf{x})=\|\mathbf{x}\|^2, ~\forall \mathbf{x}\in\mathbb{R}^{n}_{b}$, is a first integral of both vector fields, $B^{-1}\mathbf{u}, (B^{\star})^{-1}\mathbf{u}^{\star}\in\mathfrak{X}(\mathbb{R}^{n}_{b})$.
\end{proposition}
\begin{proof}
In order to prove that $N$ is a first integral of the vector field $B^{-1}\mathbf{u}$, we shall show that $\mathcal{L}_{B^{-1}\mathbf{u}}N=0$, where the notation $\mathcal{L}_{B^{-1}\mathbf{u}}$ stands for the Lie derivative along the vector field $B^{-1}\mathbf{u}$. As we will see in the following, this is equivalent to the relation $b(\mathbf{x},\mathbf{u}(\mathbf{x}))=0, ~\forall \mathbf{x}\in\mathbb{R}^{n}_{b}$. Indeed, for all $\mathbf{x}\in\mathbb{R}^{n}_{b}$ we have that
\begin{align*}
\left(\mathcal{L}_{B^{-1}\mathbf{u}}N\right)(\mathbf{x})&=\langle B^{-1}\mathbf{u}(\mathbf{x}),\nabla N(\mathbf{x})\rangle = \langle B^{-1}\mathbf{u}(\mathbf{x}),2 \mathbf{x}\rangle =\langle 2 \mathbf{x}, B^{-1}\mathbf{u}(\mathbf{x})\rangle \\
&=b(2 \mathbf{x},BB^{-1}\mathbf{u}(\mathbf{x}))=2 b(\mathbf{x},\mathbf{u}(\mathbf{x}))=0.
\end{align*}

The fact that $N$ is a first integral of the vector field $(B^{\star})^{-1}\mathbf{u}$ follows similarly, this time using the relation $b(\mathbf{u}^{\star}(\mathbf{x}),\mathbf{x})=0, ~\forall \mathbf{x}\in\mathbb{R}^{n}_{b}$. Indeed, for all $\mathbf{x}\in\mathbb{R}^{n}_{b}$ we have that
\begin{align*}
\left(\mathcal{L}_{(B^{\star})^{-1}\mathbf{u}^{\star}}N\right)(\mathbf{x})&=\left(\mathcal{L}_{(B^{-1})^{\star}\mathbf{u}^{\star}}N\right)(\mathbf{x})=\langle (B^{-1})^{\star}\mathbf{u}^{\star}(\mathbf{x}),\nabla N(\mathbf{x})\rangle = \langle (B^{-1})^{\star}\mathbf{u}^{\star}(\mathbf{x}),2 \mathbf{x}\rangle \\
&=b(B^{\star}(B^{-1})^{\star}\mathbf{u}^{\star}(\mathbf{x}),2 \mathbf{x})=b((B^{-1}B)^{\star}\mathbf{u}^{\star}(\mathbf{x}),2 \mathbf{x})=b((\operatorname{Id})^{\star}\mathbf{u}^{\star}(\mathbf{x}),2 \mathbf{x})\\
&=b(\operatorname{Id}\mathbf{u}^{\star}(\mathbf{x}),2 \mathbf{x})=2 b(\mathbf{u}^{\star}(\mathbf{x}),\mathbf{x})=0.
\end{align*}
\end{proof}

In the following we give a result which presents the dynamical properties of the vector fields $\mathbf{u},\mathbf{u}^{\star}\in\mathfrak{X}(\mathbb{R}^{n}_{b})$ relative to the geometry induced by the quadratic form associated to the geometric structure $b$.

\begin{proposition}\label{pr12}
The function $F_b:\mathbb{R}^{n}_{b}\rightarrow \mathbb{R}$, $F_{b}(\mathbf{x})=b(\mathbf{x},\mathbf{x}), ~\forall \mathbf{x}\in\mathbb{R}^{n}_{b}$, has the following properties
\begin{enumerate}
\item[(i)] $\left(\mathcal{L}_{\mathbf{u}}F_{b}\right)(\mathbf{x})=b(\mathbf{u}(\mathbf{x}),\mathbf{x}), ~ \forall \mathbf{x}\in\mathbb{R}^{n}_{b}$.
\item[(ii)] $\left(\mathcal{L}_{\mathbf{u}^{\star}}F_{b}\right)(\mathbf{x})=b(\mathbf{x},\mathbf{u}^{\star}(\mathbf{x})), ~ \forall \mathbf{x}\in\mathbb{R}^{n}_{b}$.
\item[(iii)] If $b$ is symmetric then $F_{b}$ is a first integral of the vector field $\mathbf{u}\in\mathfrak{X}(\mathbb{R}^{n}_{b})$.
\end{enumerate}
\end{proposition}
\begin{proof} Let us start by recalling that $b(\mathbf{x},\mathbf{x})=\langle \mathbf{x}, B^{-1}\mathbf{x}\rangle, ~\forall \mathbf{x}\in\mathbb{R}^{n}_{b}$. Consequently, as $B^{\top}=B^{\star}$ we obtain that $\nabla F_{b}(\mathbf{x})=B^{-1}\mathbf{x} + (B^{\star})^{-1}\mathbf{x}, ~\forall \mathbf{x}\in\mathbb{R}^{n}_{b}$.
\begin{enumerate}
\item[(i)] Using the above formula and the relation $b(\mathbf{x},\mathbf{u}(\mathbf{x}))=0, ~\forall \mathbf{x}\in\mathbb{R}^{n}_{b}$, we get that for all $\mathbf{x}\in\mathbb{R}^{n}_{b}$ the following equalities hold true
\begin{align*}
\left(\mathcal{L}_{\mathbf{u}}F_{b}\right)(\mathbf{x})&=\langle \mathbf{u}(\mathbf{x}),\nabla F_{b}(\mathbf{x})\rangle = \langle \mathbf{u}(\mathbf{x}), B^{-1}\mathbf{x} \rangle + \langle \mathbf{u}(\mathbf{x}),(B^{\star})^{-1}\mathbf{x}\rangle\\
&=b(\mathbf{u}(\mathbf{x}),\mathbf{x})+ \langle (B^{\star})^{-1}\mathbf{x},\mathbf{u}(\mathbf{x})\rangle =b(\mathbf{u}(\mathbf{x}),\mathbf{x})+b(B^{\star}(B^{\star})^{-1}\mathbf{x},\mathbf{u}(\mathbf{x}))\\
&=b(\mathbf{u}(\mathbf{x}),\mathbf{x})+ b(\mathbf{x},\mathbf{u}(\mathbf{x}))=b(\mathbf{u}(\mathbf{x}),\mathbf{x}).
\end{align*}
\item[(ii)] Using the relation $b(\mathbf{u}^{\star}(\mathbf{x}),\mathbf{x})=0, ~\forall \mathbf{x}\in\mathbb{R}^{n}_{b}$, we obtain that for all $\mathbf{x}\in\mathbb{R}^{n}_{b}$ the following equalities hold true
\begin{align*}
\left(\mathcal{L}_{\mathbf{u}^{\star}}F_{b}\right)(\mathbf{x})&=\langle \mathbf{u}^{\star}(\mathbf{x}),\nabla F_{b}(\mathbf{x})\rangle = \langle \mathbf{u}^{\star}(\mathbf{x}), B^{-1}\mathbf{x} \rangle + \langle \mathbf{u}^{\star}(\mathbf{x}),(B^{\star})^{-1}\mathbf{x}\rangle\\
&=b(\mathbf{u}^{\star}(\mathbf{x}),\mathbf{x})+ \langle (B^{\star})^{-1}\mathbf{x},\mathbf{u}^{\star}(\mathbf{x})\rangle =b(\mathbf{u}^{\star}(\mathbf{x}),\mathbf{x})+b(B^{\star}(B^{\star})^{-1}\mathbf{x},\mathbf{u}^{\star}(\mathbf{x}))\\
&=b(\mathbf{u}^{\star}(\mathbf{x}),\mathbf{x})+ b(\mathbf{x},\mathbf{u}^{\star}(\mathbf{x}))=b(\mathbf{x},\mathbf{u}^{\star}(\mathbf{x})).
\end{align*}
\item[(iii)] The proof follows directly from $(i)$ or $(ii)$, taking into account that if $b$ is symmetric, then $\mathbf{u}^{\star}=\mathbf{u}$ and hence $b(\mathbf{u}(\mathbf{x}),\mathbf{x})=b(\mathbf{x},\mathbf{u}^{\star}(\mathbf{x}))=0, ~\forall \mathbf{x}\in\mathbb{R}^{n}_{b}$.
\end{enumerate}
\end{proof}

Next result is a direct consequence of Proposition \ref{pr11} and Proposition \ref{pr12}, and gives two equivalent dynamical representations of the relation $b(\mathbf{x},\mathbf{u}(\mathbf{x}))= 0, ~\forall \mathbf{x}\in\mathbb{R}^{n}_{b}$, in the case when $b$ is symmetric.

\begin{proposition}
Let $N,F_b :\mathbb{R}^{n}_{b}\rightarrow \mathbb{R}$ be given by $N(\mathbf{x})=\|\mathbf{x}\|^2, F_b(\mathbf{x})=b(\mathbf{x},\mathbf{x}), ~\forall \mathbf{x}\in\mathbb{R}^{n}_{b}$. If $b$ is symmetric then the following statements are equivalent.
\begin{enumerate}
\item[(i)] $b(\mathbf{x},\mathbf{u}(\mathbf{x}))= 0, ~\forall \mathbf{x}\in\mathbb{R}^{n}_{b}$,
\item[(ii)] $F_{b}$ is a first integral of the vector field $\mathbf{u}\in\mathfrak{X}(\mathbb{R}^{n}_{b})$,
\item[(iii)] $N$ is a first integral of the vector field $B^{-1}\mathbf{u}\in\mathfrak{X}(\mathbb{R}^{n}_{b})$.
\end{enumerate}
\end{proposition}

Let us illustrate now the Theorem \ref{MTHM} for each of the particular cases listed in the Example \ref{expl}. We shall refer only to item $(iv)$ of Theorem \ref{MTHM}, as within the Example \ref{expl}, the geometric structure $b$ is either symmetric or skew--symmetric.

\begin{proposition}\label{prp} Let $X\in\mathfrak{X}(\mathbb{R}^{n})$ be a vector field of class $\mathcal{C}^{1}$ on $\mathbb{R}^{n}$. Let $b$ be a geometric structure on $\mathbb{R}^{n}$. Then the following assertions hold.
\begin{enumerate}
\item[(i)] If $b(\mathbf{x},\mathbf{y})=\langle \mathbf{x},\mathbf{y} \rangle, ~\forall \mathbf{x},\mathbf{y}\in\mathbb{R}^{n}$, then $\mathbb{R}^{n}_{b}$ is the canonical Euclidean $n-$dimensional vector space, and there exists a \textbf{unique global decomposition} of $X$ given by
\begin{equation}
X(\mathbf{x})=\nabla H(\mathbf{x}) + \mathbf{u}(\mathbf{x}), ~ \forall \mathbf{x}\in\mathbb{R}^{n}_{b},
\end{equation}
such that $H(\mathbf{0})=0$ and $\langle\mathbf{u}(\mathbf{x}),\mathbf{x}\rangle=\langle\mathbf{x},\mathbf{u}(\mathbf{x})\rangle=0, ~\forall \mathbf{x}\in\mathbb{R}^{n}_{b}$. (This decomposition was first given in \cite{presnov} and then extended to Riemannian manifolds in \cite{presnov2}.)

The function $N(\mathbf{x})=\|\mathbf{x}\|^2$, is a first integral of the vector field $\mathbf{u}$. 
\item[(ii)] If $n=2m$ and $b(\mathbf{x},\mathbf{y})=\omega_{0}(\mathbf{x},\mathbf{y}):=\langle \mathbf{x},\mathbb{J}_{n} \mathbf{y} \rangle, ~\forall \mathbf{x},\mathbf{y}\in\mathbb{R}^{n}$, then $\mathbb{R}^{n}_{b}$ is the canonical symplectic $n-$dimensional vector space, and there exists a \textbf{unique global decomposition} of $X$ given by
\begin{equation}
X(\mathbf{x})=X_H (\mathbf{x}) + \mathbf{u}(\mathbf{x}), ~ \forall \mathbf{x}\in\mathbb{R}^{n}_{b},
\end{equation}
such that $H(\mathbf{0})=0$ and $\omega_{0}(\mathbf{u}(\mathbf{x}),\mathbf{x})=\omega_{0}(\mathbf{x},\mathbf{u}(\mathbf{x}))=0, ~\forall \mathbf{x}\in\mathbb{R}^{n}_{b}$. 

The function $N(\mathbf{x})=\|\mathbf{x}\|^2$, is a first integral of the vector field $\mathbb{J}_{n}\mathbf{u}$. 
\item[(iii)] If $b(\mathbf{x},\mathbf{y})=\langle \mathbf{x},\mathbb{E}_{n-1,1} \mathbf{y} \rangle, ~\forall \mathbf{x},\mathbf{y}\in\mathbb{R}^{n}$, then $\mathbb{R}^{n}_{b}$ is the canonical Minkowski $n-$dimensional vector space $\mathbb{R}^{n-1,1}$, and there exists a \textbf{unique global decomposition} of $X$ given by
\begin{equation}
X(\mathbf{x})=\nabla_{n-1,1} H(\mathbf{x}) + \mathbf{u}(\mathbf{x}), ~ \forall \mathbf{x}\in\mathbb{R}^{n-1,1},
\end{equation}
such that $H(\mathbf{0})=0$ and $b(\mathbf{u}(\mathbf{x}),\mathbf{x})=b(\mathbf{x},\mathbf{u}(\mathbf{x}))=0, ~\forall \mathbf{x}\in\mathbb{R}^{n-1,1}$. 

The function $F_{n-1,1}(\mathbf{x})=\langle \mathbf{x},\mathbb{E}_{n-1,1} \mathbf{x} \rangle$ is a first integral of the vector field $\mathbf{u}$, or equivalently, the function $N(\mathbf{x})=\|\mathbf{x}\|^2$, is a first integral of the vector field $\mathbb{E}_{n-1,1} \mathbf{u}$. 
\end{enumerate} 
\end{proposition}

Let us now apply the results of the above Proposition in the case of two concrete dynamical systems, namely, the Lotka--Volterra system and respectively the Rikitake system. As both systems are defined on $\mathbb{R}^n$ (for $n=2$, respectively $n=3$), and $\mathbb{R}^n$ can be given different geometric structures, to each system apart, we want to determine explicitly its part which is compatible with the geometric structure we endow the model space with. Note that, as the dimension of the model space of the Rikitake system is odd, we cannot apply in this case the item $(ii)$ of Proposition \ref{prp}.

\begin{example}
\begin{enumerate}
\item[(1)] \textbf{The Lotka--Volterra dynamical system} provides a mathematical model to describe biological systems in which two species interact as predator versus prey. The vector field associated to the Lotka--Volterra system is given by
$$
X_{LV}(x,y):=(\alpha x-\beta xy)\partial_{x} +(\delta xy - \gamma y)\partial_{y}, ~ \forall (x,y)\in\mathbb{R}^2,
$$
where $\alpha, \beta, \gamma, \delta \geq 0$ are real parameters. For details regarding the biological interpretation of the Lotka--Volterra system, see e.g. \cite{lotka}, \cite{volterra}.  

Applying the Proposition \ref{prp} to the Lotka--Volterra system, we obtain the following geometric representations.
\begin{enumerate}
\item[(a)] 	If we endow $\mathbb{R}^{2}$ with the canonical inner product, $$b((x,y),(x',y'))=\langle(x,y),(x',y')\rangle=xx'+yy', ~\forall (x,y),(x',y')\in\mathbb{R}^{2},$$ then there exists a unique global decomposition of $X_{LV}$
\begin{equation*}
X_{LV}(\mathbf{x})=\nabla H(\mathbf{x}) + \mathbf{u}(\mathbf{x}), ~ \forall \mathbf{x}=(x,y)\in\mathbb{R}^{2}_{b},
\end{equation*}
such that $H(\mathbf{0})=0$ and $\langle\mathbf{u}(\mathbf{x}),\mathbf{x}\rangle=\langle\mathbf{x},\mathbf{u}(\mathbf{x})\rangle=0, ~\forall \mathbf{x}\in\mathbb{R}^{2}_{b}$.

More precisely, we have that
\begin{align*}
H(x,y)&=\dfrac{1}{2}\alpha x^2 -\dfrac{1}{2}\gamma y^2 +\dfrac{1}{3}xy(-\beta x+\delta y), ~\forall (x,y)\in\mathbb{R}^{2}_{b},\\
\mathbf{u}(x,y)&=\dfrac{1}{3}(\beta x +\delta y)(-y \partial_x + x\partial_y), ~\forall (x,y)\in\mathbb{R}^{2}_{b}.
\end{align*}

Note that $F_b =x^2 +y^2$ is a first integral of the vector field $\mathbf{u}$.

\item[(b)] 	If we endow $\mathbb{R}^{2}$ with the canonical symplectic form, $$b((x,y),(x',y'))=\omega_{0}((x,y),(x',y'))=xy'-x'y, ~\forall (x,y),(x',y')\in\mathbb{R}^{2},$$ then there exists a unique global decomposition of $X_{LV}$
\begin{equation*}
X_{LV}(\mathbf{x})=X_H(\mathbf{x}) + \mathbf{u}(\mathbf{x}), ~ \forall \mathbf{x}=(x,y)\in\mathbb{R}^{2}_{b},
\end{equation*}
such that $H(\mathbf{0})=0$ and $\omega_{0}(\mathbf{u}(\mathbf{x}),\mathbf{x})=\omega_{0}(\mathbf{x},\mathbf{u}(\mathbf{x}))=0, ~\forall \mathbf{x}\in\mathbb{R}^{2}_{b}$.

More precisely, we have that
\begin{align*}
H(x,y)&=xy\left[\dfrac{1}{2}\left(\alpha+\gamma\right)-\dfrac{1}{3}\delta x -\dfrac{1}{3}\beta y\right], ~\forall (x,y)\in\mathbb{R}^{2}_{b},\\
\mathbf{u}(x,y)&=\left[ \dfrac{1}{2}\left(\alpha-\gamma\right)+\dfrac{1}{3}\delta x -\dfrac{1}{3}\beta y\right](x \partial_x + y\partial_y), ~\forall (x,y)\in\mathbb{R}^{2}_{b}.
\end{align*}

\item[(c)] 	If we endow $\mathbb{R}^{2}$ with the canonical Minkowski product, $$b((x,y),(x',y'))=xx'-yy', ~\forall (x,y),(x',y')\in\mathbb{R}^{2},$$ then $\mathbb{R}^{2}_{b}=\mathbb{R}^{1,1}$, and there exists a unique global decomposition of $X_{LV}$
\begin{equation*}
X_{LV}(\mathbf{x})=\nabla_{1,1} H(\mathbf{x}) + \mathbf{u}(\mathbf{x}), ~ \forall \mathbf{x}=(x,y)\in\mathbb{R}^{1,1},
\end{equation*}
such that $H(\mathbf{0})=0$ and $b(\mathbf{u}(\mathbf{x}),\mathbf{x})=b(\mathbf{x},\mathbf{u}(\mathbf{x}))=0, ~\forall \mathbf{x}\in\mathbb{R}^{1,1}$.

More precisely, we have that
\begin{align*}
H(x,y)&=\dfrac{1}{2}\alpha x^2 +\dfrac{1}{2}\gamma y^2 -\dfrac{1}{3}xy(\beta x+\delta y), ~\forall (x,y)\in\mathbb{R}^{1,1},\\
\mathbf{u}(x,y)&=\dfrac{1}{3}\left(-\beta x +\delta y\right)(y\partial_x + x \partial_y), ~\forall (x,y)\in\mathbb{R}^{1,1}.
\end{align*}

Note that $F_{1,1}=x^2 -y^2$ is a first integral of the vector field $\mathbf{u}$.
\end{enumerate}

\item[(2)] \textbf{The Rikitake dynamical system} provides a mathematical model to describe the reversals of the Earth's magnetic field. The vector field associated to the Rikitake system is given by
$$
X_{Rik}(x,y,z):=(-\mu x +yz)\partial_{x} +[-\mu y +x(z-a)]\partial_{y} + (1-xy)\partial_{z}, ~ \forall (x,y,z)\in\mathbb{R}^3,
$$
where $\mu, a \geq 0$ are two parameters. For a physical meaning of the variables $x,y,z$ and the parameters $\mu,a$, see e.g. \cite{riki}, \cite{hardy}.  

Applying the Proposition \ref{prp} (item $(i)$ and $(iii)$) to the Rikitake system, we obtain the following geometric representations.
\begin{enumerate}
\item[(a)] 	If we endow $\mathbb{R}^{3}$ with the canonical inner product, $b((x,y,z),(x',y',z'))=\langle(x,y,z),(x',y',z')\rangle=xx'+yy'+zz', ~\forall (x,y,z),(x',y',z')\in\mathbb{R}^{3}$, then there exists a unique global decomposition of $X_{Rik}$
\begin{equation*}
X_{Rik}(\mathbf{x})=\nabla H(\mathbf{x}) + \mathbf{u}(\mathbf{x}), ~ \forall \mathbf{x}=(x,y,z)\in\mathbb{R}^{3}_{b},
\end{equation*}
such that $H(\mathbf{0})=0$ and $\langle\mathbf{u}(\mathbf{x}),\mathbf{x}\rangle=\langle\mathbf{x},\mathbf{u}(\mathbf{x})\rangle=0, ~\forall \mathbf{x}\in\mathbb{R}^{3}_{b}$.

More precisely, we have that
\begin{align*}
H(x,y,z)&=-\dfrac{\mu}{2}x^2 -\dfrac{\mu}{2}y^2 + \dfrac{1}{3}xyz -\dfrac{a}{2}xy+z, ~\forall (x,y,z)\in\mathbb{R}^{3}_{b},\\
\mathbf{u}(x,y,z)&=\left(\dfrac{2}{3}yz +\dfrac{a}{2}y\right)\partial_x +\left(\dfrac{2}{3}xz -\dfrac{a}{2}x\right)\partial_y -\dfrac{4}{3}xy\partial_{z}, ~\forall (x,y,z)\in\mathbb{R}^{3}_{b}.
\end{align*}

Note that $F_b =x^2 +y^2 +z^2$ is a first integral of the vector field $\mathbf{u}$.

\item[(b)] 	If we endow $\mathbb{R}^{3}$ with the canonical Minkowski product, $b((x,y,z),(x',y',z'))=xx'+yy'-zz', ~\forall (x,y,z),(x',y',z')\in\mathbb{R}^{3}$, then $\mathbb{R}^{3}_{b}=\mathbb{R}^{2,1}$, and there exists a unique global decomposition of $X_{Rik}$
\begin{equation*}
X_{Rik}(\mathbf{x})=\nabla_{2,1} H(\mathbf{x}) + \mathbf{u}(\mathbf{x}), ~ \forall \mathbf{x}=(x,y,z)\in\mathbb{R}^{2,1},
\end{equation*}
such that $H(\mathbf{0})=0$ and $b(\mathbf{u}(\mathbf{x}),\mathbf{x})=b(\mathbf{x},\mathbf{u}(\mathbf{x}))=0, ~\forall \mathbf{x}\in\mathbb{R}^{2,1}$.

More precisely, we have that
\begin{align*}
H(x,y,z)&=-\dfrac{\mu}{2}x^2 -\dfrac{\mu}{2}y^2 + xyz -\dfrac{a}{2}xy -z, ~\forall (x,y,z)\in\mathbb{R}^{2,1},\\
\mathbf{u}(x,y,z)&=\dfrac{a}{2}\left(y\partial_x - x \partial_y \right), ~\forall (x,y,z)\in\mathbb{R}^{2,1}.
\end{align*}

Note that for every $\psi\in\mathcal{C}^{1}(\mathbb{R},\mathbb{R})$, the function $F_{\psi}=x^2 +y^2 +\psi(z)$ is a first integral of the vector field $\mathbf{u}$; for $\psi=- z^2$, we obtain $F_{\psi}=F_{2,1}=x^2+y^2-z^2$.

\textbf{In contrast with the incompatibility between the Rikitake system and the Euclidean geometry of its model space, for $a=0$, the vector field $\mathbf{u}$ vanishes identically, and consequently the Rikitake system becomes a Minkowski gradient system, for every $\mu\geq 0$}.
\end{enumerate}
\end{enumerate}
\end{example}

\section{On topological conjugacy of vector fields on $\mathbb{R}^{n}$}

The aim of this short section is to apply the geometric decomposition of vector fields given in Theorem \ref{MTHM} in order to give a criterion to decide topological conjugacy of vector fields of class $\mathcal{C}^1$ on $\mathbb{R}^{n}$ based on the topological conjugacy of the corresponding parts given by the splitting introduced in Theorem \ref{MTHM}. From now on we shall assume that all vector fields are complete.

Before stating the main result of this section (which is an extension of Theorem 3.1 from \cite{anas}), let us recall that two $\mathcal{C}^1$ vector fields on $\mathbb{R}^{n}$ are called \textit{topologically equivalent} if there exists a homeomorphism which takes orbits to orbits and preserves also their orientation. If moreover the homeomorphism preserves also the parameterizations of the orbits then the two vector fields are said \textit{topologically conjugate}; for details regarding topological conjugacy and related topics see e.g. \cite{robinson}.

Let us state now the main result of this section, which is a generalization of Theorem 3.1 from \cite{anas}.

\begin{theorem}\label{MTHM2} Let $X_1, X_2 \in\mathfrak{X}(\mathbb{R}^{n})$ be two vector fields of class $\mathcal{C}^1$ on $\mathbb{R}^{n}$. Let $b_1$, $b_2$ be two inner products on $\mathbb{R}^n$, and let $X_1 = X_{1}^{g,b_1} +X_{1}^{r,b_1}$, $X_2 = X_{2}^{g,b_2} +X_{2}^{r,b_2}$ be the associated geometric decompositions of type \eqref{decvf}. 

Assume that both vector fields, $X_{1}^{g,b_1}$, $X_{2}^{g,b_2}$, admit a unique equilibrium point, $\mathbf{x}=\mathbf{0}$, which is moreover globally attracting (repelling).

Then $X_1$ is topologically conjugated to $X_2$, if $X_{1}^{r,b_1}$ and $X_{2}^{r,b_2}$ are topologically conjugated.  
\end{theorem}
\begin{proof}
Let $B_1, B_2 \in\operatorname{GL}(n,\mathbb{R})$ be such that $(b_1, B_1)$, $(b_2, B_2)$ are the geometric pairs associated to $b_1, b_2$. Recall from Proposition \ref{pr11} that the vector fields $(B_1)^{-1} X_{1}^{r,b_1}$ and $(B_2)^{-1} X_{2}^{r,b_2}$ are sphere--preserving, i.e. they are tangent to each sphere with center in the origin and arbitrary radius. As each continuous sphere--preserving vector field in $\mathbb{R}^n$ admits the origin as equilibrium, it follows that the origin is also an equilibrium state of the vector fields $X_{1}^{r,b_1}$, $X_{2}^{r,b_2}$. Since by hypothesis, the origin is the unique equilibrium point of the vector fields $X_{1}^{g,b_1}$, $X_{2}^{g,b_2}$, we obtain that $\mathbf{x}_{e}=\mathbf{0}$ is an equilibrium point of $X_1$ and $X_2$. 

The rest of the proof follows mimetically the proof of Theorem 3.1 from \cite{anas}. Nevertheless, for the sake of completeness we shall recall from \cite{anas} only the main steps of the proof (without full computations), adapted to our situation. In the following we shall consider two cases, namely, first, if the origin is a globally attracting equilibrium point of $X_{1}^{g,b_1}$, $X_{2}^{g,b_2}$, and second, if the origin is a globally repelling equilibrium point of $X_{1}^{g,b_1}$, $X_{2}^{g,b_2}$. We shall discuss only the first case, the second case being treated similarly.

For $i\in\{1,2\}$, we denote by $H_{i}^{b_i}:\mathbb{R}^{n}\rightarrow \mathbb{R}$, the potential function associated to the gradient vector field $X_{i}^{g,b_i}$ (i.e. $X_{i}^{g,b_i}=\nabla_{b_i}H_{i}^{b_i}$), whose corresponding (global) flow we denote by $\Phi_{i}^{g,b_i}:\mathbb{R}\times\mathbb{R}^{n}\rightarrow \mathbb{R}^{n}$.

As the origin is a globally attracting equilibrium point of $X_{1}^{g,b_1}$ and $X_{2}^{g,b_2}$, there exist $h_{1}^{b_1},h_{2}^{b_2}\in\mathbb{R}$ such that for $i\in\{1,2\}$, the level set $\Sigma_{i}^{b_i}:=\{\mathbf{x}\in\mathbb{R}^{n}: ~H_{i}^{b_i}(\mathbf{x})=h_{i}^{b_i} \}$ is a codimension--one submanifold of $\mathbb{R}^n$ which does not contain the origin, and moreover the mapping $F_{i}^{b_i}:\mathbb{R}\times\Sigma_{i}^{b_i}\rightarrow \mathbb{R}^{n}\setminus\{\mathbf{0}\}$, given by $F_{i}^{b_i}(t,\mathbf{u})= \Phi_{i}^{g,b_i}(t,\mathbf{u})$, is a diffeomorphism.

Next, a direct computation shows that the vector fields $(F_{2}^{b_2})^{\star}X_1$, $(F_{2}^{b_2})^{\star}X_2 \in\mathfrak{X}(\mathbb{R}\times\Sigma_{2}^{b_2})$ are topologically conjugated (see \cite{anas} for computational details), or similarly, the vector fields $(F_{1}^{b_1})^{\star}X_1$, $(F_{1}^{b_1})^{\star}X_2 \in\mathfrak{X}(\mathbb{R}\times\Sigma_{1}^{b_1})$ are topologically conjugated. Thus, there exists a homeomorphism $h_{2}^{b_2}:\mathbb{R}\times\Sigma_{2}^{b_2} \rightarrow \mathbb{R}\times\Sigma_{2}^{b_2}$ such that
\begin{equation}\label{top1}
(h_{2}^{b_2})^{-1}\circ [(F_{2}^{b_2})^{\star}X_1]^{\tau} \circ h_{2}^{b_2} = [(F_{2}^{b_2})^{\star}X_2]^{\tau},
\end{equation}
where the notation $X^{\tau}$ stands for the flow of the vector field $X$.

If one denotes $F_{2}^{b_2}\circ h_{2}^{b_2} \circ (F_{2}^{b_2})^{-1}=:K_{2}^{b_2}:\mathbb{R}^{n}\setminus\{\mathbf{0}\}\rightarrow \mathbb{R}^{n}\setminus\{\mathbf{0}\}$, the relation \eqref{top1} becomes $(K_{2}^{b_2})^{-1}\circ X_1^{\tau}\circ K_{2}^{b_2}= X_2^{\tau}$, and hence $X_1$ and $X_2$ are topologically conjugated as vector fields on $\mathbb{R}^{n}\setminus\{\mathbf{0}\}$.  

The global topological conjugacy of $X_1$ and $X_2$ follows if we extend by continuity the homeomorphism $K_{2}^{b_2}$ at $\mathbf{x}=\mathbf{0}$, by setting $K_{2}^{b_2}(\mathbf{0})=\mathbf{0}$. 

\end{proof}

\begin{remark}
\begin{enumerate}
\item[(1)] An advantage of Theorem \ref{MTHM2} is that for a fixed pair of vector fields, $X_1, X_2$, we have the freedom to choose the inner products $b_1, b_2$.
\item[(2)] The hypothesis of Theorem \ref{MTHM2} does not hold if any of the inner products $b_1$, $b_2$ is replaced by a skew--symmetric geometric structure (i.e. symplectic form), as in this case, the associated gradient--like vector field becomes Hamiltonian, and consequently the origin cannot be an asymptotically stable equilibrium point of this vector field.
\item[(3)] In Theorem \ref{MTHM2}, if any of the inner products $b_1$, $b_2$ is replaced by a symmetric geometric structure (denoted by $b$) with signature $(p,q)$ such that $pq \neq 0$, then the (appropriate) restriction of the (global) flow of the associated gradient--like vector field $\nabla_{b}H^{b}$, does not necessarily provide a diffeomorphism between $\mathbb{R}\times\Sigma^{b}$ and $\mathbb{R}^{n}\setminus\{\mathbf{0}\}$, where $\Sigma^{b}$ stands for a level set of $H^{b}$ such that $\mathbf{0}\notin\Sigma^{b}$.

A simple example is given by the vector field $X=(-x+y)\partial_{x}+(x-y)\partial_{y}\in\mathfrak{X}(\mathbb{R}^{1,1})$. In this case it follows that the gradient--like part of $X$ with respect to the canonical Minkowski product is $\nabla_{1,1}H= -x\partial_{x}-y\partial_{y}$, where $H=\dfrac{1}{2}(-x^2 +y^2)$. Hence, the origin is the unique equilibrium of $\nabla_{1,1}H$, which is moreover globally asymptotically stable. Nevertheless, for each fixed $h\neq 0$, the image of $\mathbb{R}\times H^{-1}(\{h\})$ through the flow of $\nabla_{1,1}H$ is strictly included in $\mathbb{R}^2 \setminus\{(0,0)\}$.
\end{enumerate} 
\end{remark}


\bigskip
\bigskip

\noindent {\sc R.M. Tudoran}\\
West University of Timi\c soara\\
Faculty of Mathematics and Computer Science\\
Department of Mathematics\\
Blvd. Vasile P\^arvan, No. 4\\
300223 - Timi\c soara, Rom\^ania.\\
E-mail: {\sf razvan.tudoran@e-uvt.ro}\\
\medskip

\end{document}